\documentclass{article}

\usepackage[utf8]{inputenc} \usepackage[T1]{fontenc}    \usepackage[english]{babel}

\usepackage{amsfonts}
\usepackage{nicefrac}       \usepackage{microtype}      \usepackage{lmodern}
\usepackage{amssymb,amsmath,amsthm}

\usepackage{stmaryrd}
\SetSymbolFont{stmry}{bold}{U}{stmry}{m}{n}

\usepackage{bbm,bm}
\usepackage{latexsym}
\usepackage{xcolor}

\usepackage{enumerate}
\usepackage{verbatim}
\usepackage{booktabs}       

\usepackage{url}            \usepackage[colorlinks=true]{hyperref}
\usepackage[numbers,sort&compress,square,comma]{natbib}
\usepackage[capitalise,noabbrev]{cleveref}

\usepackage{geometry}
\usepackage[short]{optidef}
\usepackage{algorithmic,algorithm}

\usepackage{thmtools}

\declaretheorem{theorem}
\declaretheorem{corollary}
\declaretheorem{lemma}

\declaretheoremstyle[qed=$\square$]{definitionwithend}

\declaretheorem[style=definitionwithend]{example}
\declaretheorem[style=definitionwithend]{remark}

\crefname{assumption}{Assumption}{Assumptions}
\crefname{conjecture}{Conjecture}{Conjectures}

\newcommand{\abs}[1]{\ensuremath{\left\lvert #1 \right\rvert}}

\newcommand{\norm}[1]{\ensuremath{\left\lVert #1 \right\rVert}}
\newcommand{\ip}[1]{\ensuremath{\left\langle #1 \right\rangle}}

\let\emptyset\varnothing
\newcommand{\set}[1]{\left\{#1\right\}}

\def\B{{\mathbb{B}}}

\def\R{{\mathbb{R}}}

\def\cF{{\cal F}}

\DeclareMathOperator*{\argmin}{arg\,min}

\DeclareMathOperator{\spann}{span}

\DeclareMathOperator{\dom}{dom}

 \usepackage{soul}

\DeclareMathOperator{\supp}{supp}
\DeclareMathOperator{\dist}{dist}

\newcommand{\Polyak}{\texttt{PolyakGD}}
\newcommand{\Geom}{\texttt{GeomGD}}
\newcommand{\Decay}{\texttt{DecayGD}}

\begin{document}

\title{Optimal Subgradient Methods for Lipschitz Convex Optimization with Error Bounds}
\author{
    Alex L.\ Wang
    \thanks{Purdue University. West Lafayette, IN, USA(\texttt{wang5984@purdue.edu})}
}
\date{\today}
\maketitle

\begin{abstract}
We study the iteration complexity of Lipschitz convex optimization problems satisfying a general error bound. We show that for this class of problems, subgradient descent with either Polyak stepsizes or decaying stepsizes achieves minimax optimal convergence guarantees for decreasing distance-to-optimality.
The main contribution is a novel lower-bounding argument that produces hard functions simultaneously satisfying zero-chain conditions and global error bounds. \end{abstract}

\section{Introduction}

This paper considers 
the iteration complexity of
Lipschitz convex optimization with error bounds. By ``error bound'', we mean that the objective function $f:\R^d\to\R$ satisfies the following growth property:
\begin{equation*}
    f(x) - f_\star \geq h\left(\dist(x, \argmin f)\right)
\end{equation*}
on a neighborhood of $\argmin f$.
Here, $f_\star$ denotes the minimum value of $f$, $\argmin f$ denotes the minimizers of $f$ (nonempty by assumption), and the function $h$ parameterizes the error bound. All norms, including the norm implicit in $\dist(x,\argmin f)$, in this paper are Euclidean.

Error bounds are widespread in the optimization literature~\cite{hoffman2003approximate,zhou2017unified,bolte2007lojasiewicz,bolte2017error}. Indeed, it is known~\cite{bochnak2013real} that for any semialgebraic $f$, 
the following H\"olderian error bound
\begin{equation}
    \label{eq:holder}
    f(x) - f_\star \geq c\dist(x, \argmin f)^{1/\theta}
\end{equation}
holds locally for some choice of $c>0$ and $\theta\in(0,1]$.
This property is also known as the \L{}ojasiewicz inequality~\cite{bolte2007lojasiewicz}.

Simultaneously, a variety of methods have been developed to exploit these error bounds to \emph{accelerate} subgradient methods in different settings. These include the restarting strategies~\cite{yang2018rsg,roulet2017sharpness,necoara2019linear,ding2023sharpness} first suggested by Nemirovski and Nesterov~\cite{nemirovski1985optimal}, decaying stepsize schedules~\cite{johnstone2020faster,goffin1977convergence,eremin1965relaxation}, and ``Polyak'' stepsizes~\cite{polyak1969minimization}.
To illustrate this acceleration, suppose $f$ is $L$-Lipschitz and $\mu$-weakly sharp, i.e.,  $f(x)- f_\star\geq \mu \dist(x,\argmin f)$ for some $\mu\in(0,L]$, and suppose $\dist(x_0,\argmin f)\leq D$.
In this setting, it is not hard to show (see \cref{sec:upper}) that subgradient descent with
the Polyak stepsize~\cite{polyak1969minimization} 
\begin{equation*}
    \Polyak{}:\qquad x_n = x_{n-1} - \frac{f_{n-1} - f_\star}{\norm{g_{n-1}}^2}g_{n-1}
\end{equation*}
decreases the distance-to-optimality at a linear rate:
\begin{equation*}
 \dist(x_n, \argmin f) \le \left(1-\frac{\mu^2}{L^2}\right)^{n/2} D.   
\end{equation*}
If $\mu\in\left(0,L/\sqrt{2}\right]$, then the same convergence guarantee holds for subgradient descent with geometrically decaying stepsizes~\cite{goffin1977convergence,eremin1965relaxation,johnstone2020faster},
\begin{equation*}
    \Geom{}:\qquad x_n = x_{n-1} - \left[(1-\mu^2/L^2)^{(n-1)/2}\frac{\mu D}{L^2}\right]\,g_{n-1}.
\end{equation*}
Here and throughout, we use the shorthand $f_i$ and $g_i$ to denote $f(x_i)$ and any element of $\partial f(x_i)$.

Despite the long line of work on optimization algorithms in the presence of error bounds, lower-bounding arguments are to-date missing from the literature. A core difficulty towards producing such constructions is enforcing global error bounds for general functions $h$. Indeed, a common paradigm for building lower bounds\footnote{This is the technique used to construct optimal lower bounds in the nonsmooth convex, smooth convex, and smooth strongly convex settings.} defines a hard instance only locally at finitely many points before invoking interpolation theory to extend the hard instance globally~\cite{drori2016optimal,drori2017exact,drori2022oracle,nemirovskij1983problem}. 
Unfortunately, these tools provide little control for enforcing global error bounds (outside of the smooth strongly convex setting~\cite{drori2022oracle}).
Even in the simplest setting of weakly-sharp Lipschitz convex optimization, lower bounds exactly matching the guarantees of \Polyak{} or \Geom{} have yet to be exhibited.
Our key contribution in this paper is to provide such a lower-bounding construction for any choice of $h$, thereby showing that both \Polyak{} and 
a generalization of \Geom{} are exactly minimax optimal.
Our lower-bounding construction fully specifies the hard instance on finitely many \emph{subspaces}, before ``gluing'' these pieces togehter.

\subsection{Setup, contributions, and outline}
Let $L,D>0$ and suppose $h:[0,D]\to\R_+$ is an $L$-Lipschitz, convex, increasing function vanishing at $0$.
Let $\cF_{h,L,D}$ denote the set of instances $(f,x_0)$ satisfying
\begin{itemize}
    \item $f:\R^d\to\R$ is $L$-Lipschitz and convex, with minimizers $\argmin f$ and minimum value $f_\star$,
    \item $f(x) - f_\star \geq h(\dist(x, \argmin f))$ for all $x$ satisfying $\dist(x, \argmin f)\leq D$, and
    \item $\dist(x_0, \argmin f)\leq D$.
\end{itemize}

\Cref{sec:upper} collects minor extensions and strengthenings of known algorithmic results~\cite{polyak1969minimization,johnstone2020faster} on the convergence rate of \Polyak{} and subgradient methods with decaying stepsizes.
First, we inductively define a sequence of ``rates'' $\Delta_0,\Delta_1,\dots$ and prove that \Polyak{} produces iterates satisfying $\dist(x_n, \argmin f)\leq \Delta_n$ for all $n=0,1,\dots$ (see \cref{thm:upper}).
This extends the convergence results of~\cite{polyak1969minimization} to the setting of general error bounds $h$.
Next, we define a decaying stepsize algorithm, \Decay{}, which generalizes \Geom{} to the setting of general $h$.
We show that if $h$ is $L/\sqrt{2}$-Lipschitz, then \Decay{} produces iterates with the same convergence guarantees as \Polyak{}, i.e., $\dist(x_n, \argmin f)\leq \Delta_n$.
This simultaneously strengthens the convergence guarantees and weakens the compactness assumption of \cite{johnstone2020faster} in the H\"olderian error bound setting.

\Cref{sec:lower} is our main contribution and constructs a hard instance depending on $h,L,D$ for which \emph{every} subgradient-span method, even with full knowledge of $h,L,D$ \emph{and} $f_\star$, has all iterates $\dist(x_n,\argmin f)\geq \Delta_n$ (see \cref{thm:lower}). 
We deduce that, \Polyak{} is minimax optimal in the presence of any error bound $h$ and that \Decay{} is exactly minimax optimal as long as $h$ is $L/\sqrt{2}$-Lipschitz.

\section{Upper bounds}
\label{sec:upper}
 
\Polyak{} and \Decay{} iteratively set
\begin{equation*}
    x_{n} = x_{n-1} - \eta_{n-1} g_{n-1},
\end{equation*}
where $g_{n-1}\in\partial f(x_{n-1})$ and the stepsize is given by
\begin{equation*}
\eta_{n-1} = \frac{f_{n-1}-f_\star}{\norm{g_{n-1}}^2}\qquad\text{and}\qquad
\eta_{n-1} = \frac{h(\Delta_{n-1})}{L^2}
\end{equation*}
in \Polyak{} and \Decay{} respectively. The sequence
$\Delta_0,\Delta_1,\dots$ is defined via the recursion:  $\Delta_0=D$ and 
\begin{equation}
    \label{eq:Delta}
    \Delta_n \coloneqq \sqrt{\Delta_{n-1}^2 - \frac{h(\Delta_{n-1})^2}{L^2}} \quad\forall n \geq 1.
\end{equation}
Since $h$ is $L$-Lipschitz, we deduce that the sequence $\Delta_0,\Delta_1,\dots$ is well-defined, non-increasing, and converges to zero.
Note that computing the \Polyak{} stepsize requires knowledge of $f_\star$, whereas computing the \Decay{} stepsize requires knowledge of $h, L, D$.
In the H\"olderian error bound setting, i.e., $h(t) = ct^{1/\theta}$, the \Decay{} stepsizes match the ``decaying'' stepsizes of Johnstone and Moulin~\cite{johnstone2020faster} up to lower order terms.

\begin{remark}
The following derivation motivates the recursive definition of $\Delta_n$ and the \Decay{} stepsize. Suppose $\argmin f = \set{x_\star}$ is unique and $\norm{x_{n-1} - x_\star}=\Delta_{n-1}$ \emph{exactly}. By convexity and the error bound, for any $g_{n-1}\in \partial f(x_{n-1})$, we have that
$\ip{g_{n-1},x_{n-1}-x_\star}\geq f_{n-1}-f_\star \geq h(\Delta_{n-1})$. Thus, for any $\eta_{n-1}\geq 0$,
\begin{align*}
    \norm{x_{n-1}-\eta_{n-1} g_{n-1} - x_\star}^2 &= \norm{x_{n-1} - x_\star}^2 -2\eta_{n-1} \ip{g_{n-1}, x_{n-1} - x_\star} + \eta_{n-1}^2\norm{g_{n-1}}^2\\
    &\leq \Delta_{n-1}^2 -\eta_{n-1} h(\Delta_{n-1}) + \eta_{n-1}^2L^2.
\end{align*}
The final expression achieves its minimum value $\Delta_{n-1}^2 - \frac{h(\Delta_{n-1})^2}{L^2}$ at $\eta_{n-1} = \frac{h(\Delta_{n-1})}{L^2}$. 
This matches exactly the definition of $\Delta_n^2$ and the \Decay{} stepsize.
However, this alone is not enough to prove that $\norm{x_n-x_\star} \leq \Delta_n$ for all $n$ as the above calculation only works assuming that $\norm{x_{n-1}-x_\star} = \Delta_{n-1}$ exactly.
\end{remark}

For context, the following lemma states asymptotic behavior of the sequence $\Delta_n$ in the H\"olderian error bound setting. Its proof is standard and left to the appendix.

\begin{lemma}
\label{lem:heb_asymptotics}
Suppose $L,D>0$ and that $h(t) = c t^{1/\theta}$ for some $c>0$ and $\theta\in(0,1]$. Assume that $h$ is $L$-Lipschitz, or equivalently $D^{(1-\theta)/\theta}\leq \frac{L\theta}{c}$. Then, the sequence $\Delta_0,\Delta_1,\dots$ has the following asymptotic behavior
\begin{equation*}
    \Delta_n \equiv \begin{cases}
    \left(\frac{c^2}{L^2}\frac{1-\theta}{\theta} n\right)^{-\theta/2(1-\theta)} & \text{if }\theta <1\\
    \left(1-\frac{c^2}{L^2}\right)^{n/2} D & \text{if } \theta = 1
    \end{cases}.
\end{equation*}
Here, $\equiv$ denotes equality up to lower order terms as $n\to\infty$.
\end{lemma}

Remarkably, \Polyak{} is able to achieve this convergence guarantee even without knowledge of $h,L,D$. Under an additional Lipschitzness assumption on $h$, we will show that \Decay{} also achieves this guarantee.
\begin{theorem}
    \label{thm:upper}
Suppose $L,D >0$, and $h:[0,D]\to\R_+$ is $L$-Lipschitz, convex, increasing, and vanishes at $0$. Let $(f,x_0)\in \cF_{h,L,D}$ and let $x_0,x_1,\dots$ denote the iterates produced by \Polyak{}. Then, for all $n\geq 0$,
\begin{equation*}
    \dist(x_n, \argmin f) \leq \Delta_n.
\end{equation*}
If $h$ is $\frac{L}{\sqrt{2}}$-Lipschitz, then the same bound holds for the iterates $x_0,x_1,\dots$ produced by \Decay{}.
\end{theorem}
\begin{proof}
We prove this via induction. The base case, $\dist(x_0, \argmin f)\leq \Delta_0$, holds by the assumption that $(f,x_0) \in \cF_{h,L,D}$. Now, let $n\geq 1$ and let $x_\star$ so that $\Delta\coloneqq \norm{x_{n-1} - x_\star} = \dist(x_{n-1},\argmin f)\leq \Delta_{n-1}$. 
Then,
\begin{align}
    \norm{x_n - x_\star}^2 &= \norm{x_{n-1} - \eta_{n-1}g_{n-1} - x_\star}^2\nonumber\\
    &= \Delta^2 - 2\eta_{n-1}\ip{x_{n-1} - x_\star, g_{n-1}} + \eta_{n-1}^2\norm{g_{n-1}}^2.    \label{eq:norm_expansion}
\end{align}
The remainder of the proof depends on the choice of \Polyak{} or \Decay{} stepsize.

Suppose we pick the \Polyak{} stepsize and that $h$ is $L$-Lipschitz.
As $f$ is convex and satisfies the error bound, we have that
$\ip{g_{n-1}, x_{n-1} - x_\star} \geq f_{n-1} - f_\star \geq h(\dist(x_{n-1},\argmin f)) = h(\Delta)$.
Plugging $\eta_{n-1} = \frac{f_{n-1}-f_\star}{\norm{g_{n-1}}^2}$ and this inequality into \eqref{eq:norm_expansion} gives
\begin{align}
    \norm{x_n-x_\star}^2 
    \leq \Delta^2 - \frac{h(\Delta)^2}{L^2}\leq \max_{\Delta\in[0,\Delta_{n-1}]}\Delta^2 - \frac{h(\Delta)^2}{L^2}.
\end{align}
As $h$ is $L$-Lipschitz, we have that $\Delta^2 - \frac{h(\Delta)^2}{L^2}$ is a nondecreasing function so that $\norm{x_n - x_\star}^2 \leq \Delta_n^2$.

Now, suppose we pick the \Decay{} stepsize and that $h$ is $\frac{L}{\sqrt{2}}$-Lipschitz. Again, as $f$ is convex and satisfies the error bound, we have that
$\ip{g_{n-1}, x_{n-1} - x_\star} \geq f_{n-1} - f_\star \geq h(\dist(x_{n-1},\argmin f)) = h(\Delta)$.
Plugging $\eta_{n-1} = \frac{h(\Delta_{n-1})}{L^2}$ and this inequality into \eqref{eq:norm_expansion} gives
\begin{align*}
    \norm{x_n - x_\star}^2 &= \Delta^2 - 2\frac{h(\Delta_{n-1})h(\Delta)}{L^2} + \frac{h(\Delta_{n-1})^2}{L^4}\norm{g_{n-1}}^2\\
    &\leq \Delta^2 - 2\frac{h(\Delta_{n-1})h(\Delta)}{L^2} + \frac{h(\Delta_{n-1})^2}{L^2}\\
    &\leq \max_{\Delta\in[0,\Delta_{n-1}]}\Delta^2 - 2\frac{h(\Delta_{n-1})h(\Delta)}{L^2} + \frac{h(\Delta_{n-1})^2}{L^2}\\
    &\leq \max_{\Delta\in[0,\Delta_{n-1}]}\Delta^2 - 2\frac{h(\Delta_{n-1})\left[h(\Delta_{n-1}) + \frac{L}{\sqrt{2}}(\Delta - \Delta_{n-1})\right]}{L^2} + \frac{h(\Delta_{n-1})^2}{L^2}.
\end{align*}
Here, the second line bounds $\norm{g_{n-1}}\leq L$ and the last line lower bounds $h(\Delta)$ on the interval $[0,\Delta_{n-1}]$ using the $\frac{L}{\sqrt{2}}$-Lipschitzness of $h$.
Note that the objective function in the final line is convex quadratic in $\Delta$, thus is maximized at one of its endpoints. We deduce
\begin{align*}
    \norm{x_n-x_\star}^2 &\leq  \max\left(\frac{\sqrt{2}}{L}h(\Delta_{n-1})\Delta_{n-1} - \frac{h(\Delta_{n-1})^2}{L^2}, \Delta_{n-1}^2 - \frac{h(\Delta_{n-1})^2}{L^2}\right)\\
    &= \Delta_{n-1}^2 - \frac{h(\Delta_{n-1})^2}{L^2} = \Delta_n^2.
\end{align*}
The final line uses $h(\Delta_{n-1}) \leq \frac{L}{\sqrt{2}} \Delta_{n-1}$.
\end{proof}

\begin{remark}
    \label{rem:lips}
The requirement that $h$ is $\frac{L}{\sqrt{2}}$-Lipschitz for \Decay{} can be replaced by any condition on $h$ 
ensuring that
\begin{equation*}
    \max_{\Delta\in[0,D']}\Delta^2 - 2\frac{h(D')h(\Delta)}{L^2} + \frac{h(D')^2}{L^2}
\end{equation*}
is achieved at its rightmost endpoint for any $D'\in[0,D]$. However, this requirement cannot be removed entirely and is not just an artifact of our analysis: For some choices of $h$, as we will see below, no subgradient method without knowledge of $f_\star$ can match the performance of \Polyak{}. This rules out the possibility of constructing exact lower bounds
for subgradient methods without knowledge of $f_\star$
using zero-chain constructions (defined in \cref{sec:lower}) for these choices of $h$.
\end{remark}
\begin{example}
Consider taking
$L=1$, $D=1$ and $h(t) = \mu t$ where $\mu = \frac{1+\epsilon}{\sqrt{2}}$. Consider two different univariate instances:
The first instance sets $f(x) = \mu\abs{x}$ and $x_0= 0$.
The second instance sets $f(x) = \mu\abs{x-1} - \mu$ and $x_0=0$. Both instances belong to $\cF_{h,L,D}$.
Now, $f(x_0) = 0$ for both instances and furthermore $-\mu\in\partial f(x)$ for both instances. 
Thus, any algorithm receiving $f_0=0$ and $g_0 = -\mu$ is unable to distinguish between these two instances.
Since the minimizers of the two instances are unit distance apart, no algorithm knowing only $h,L,D$ can guarantee
\begin{equation*}
    \dist(x_1,\argmin f) < \frac{1}{2}.
\end{equation*}
On the other hand, \Polyak{} guarantees $\dist(x_1,\argmin f) \leq \sqrt{1 - \mu^2} <1/2$. This is only possible because \Polyak{} can distinguish between the two instances via their minimum values.
\end{example}

\section{Lower bounds}
\label{sec:lower}

We now turn to lower bounds and will show that \Polyak{} is minimax optimal for the class $\cF_{h,L,D}$. In particular, under suitable conditions on $h$, \Decay{} is also minimax optimal. The following lower bound is stated in terms of the sequence $\Delta_0,\Delta_1,\dots$ defined in \eqref{eq:Delta}.

\begin{theorem}
    \label{thm:lower}
Suppose $d\geq N+1$.
Suppose $L,D>0$ and $h:[0,D]\to\R_+$ is $L$-Lipschitz, convex, increasing, and vanishes at 0. There exists an instance $(f_\textup{hard},x_0=0)\in \cF_{h,L,D}$ with $f_\textup{hard}(x_\star) = 0$ and a resisting oracle for $x\mapsto g\in\partial f_\textup{hard}(x)$ so that for any choice of $x_1,\dots, x_N$ satisfying the subgradient-span condition, $x_n\in x_0 + \spann\set{g_0,\dots,g_{n-1}}$, it holds that
\begin{equation*}
    \dist(x_n, \argmin f_\textup{hard}) \geq \Delta_n \qquad\forall n=0,\dots,N.
\end{equation*}
\end{theorem}

The subgradient-span condition is relatively benign and holds for all standard Euclidean methods utilizing a subgradient oracle such as subgradient descent, bundle methods, cutting plane methods, etc.

We will without loss of generality assume $L=D =1$ (the general case follows by a rescaling argument) and that $d = N+1$ (the general case follows by a padding argument).

The function $f_\textup{hard}:\R^{N+1}\to\R$ that we will construct will have a unique minimizer, $x_\star\in\R^{N+1}$ where
\begin{equation*}
    (x_\star)_{i} = \begin{cases}
    -h(\Delta_{i-1}) &\text{if }i \in[1,N]\\
    -\Delta_N &\text{if }i = N+1
    \end{cases}.
\end{equation*}
By a telescoping argument, it holds that for all $n=0,\dots,N$,
\begin{equation}
    \label{eq:distance}
    \sum_{i=n+1}^{N+1} (x_\star)_i^2 = \sum_{i=n}^{N-1} h(\Delta_i)^2 + \Delta_N^2 = \sum_{i=n}^{N-1} (\Delta_{i}^2 - \Delta_{i+1}^2) + \Delta_N^2  = \Delta_{n}^2.
\end{equation}

Abusing notation, we let $h:\R_+\to\R_+$ denote an $1$-Lipschitz, convex, increasing extension of $h$. Set $f(x) = h(\norm{x-x_\star})$ and note that $f(x)$ is convex and $1$-Lipschitz. Furthermore, since $h$ is increasing, $x_\star$ is the unique minimizer of $f$. Additionally, $f$ satisfies the error bound
\begin{equation*}
    f(x) - f_\star \geq h(\norm{x - x_\star})\quad\text{on}\quad\set{x\in\R^{N+1}:\, \norm{x - x_\star}\leq D}.
\end{equation*}

We will ``obfuscate'' $f(x)$ by defining
\begin{equation*}
 f_\textup{hard}(x) = \max\left(\phi_1(x),\phi_2(x),\dots,\phi_{N+1}(x), f(x)\right)   ,
\end{equation*}
where each $\phi_{n+1}(x)$ is a carefully chosen extension of a restriction of $f(x)$ to a coordinate space.

For $n=0,\dots,N$, let $\Pi_n:\R^{N+1}\to\R^n$ denote restriction onto the first $n$ coordinates, and let $\Pi_n^*$ denote its adjoint.
Let $f_n:\R^n\to\R$ denote the restriction of $f$ to the first $n$ coordinates and let $z_n$ denote \emph{its} minimzer, i.e.,
\begin{equation*}
f_n(x) \coloneqq f(\Pi_n^* x)\qquad\text{and}\qquad
z_n\coloneqq \Pi_n x_\star.
\end{equation*}
Note that $f_n$ is a $1$-Lipschitz convex function in $n$ dimensions. When extending $f_n$ to one additional coordinate, we will use the additional coordinate to increase the norm of the subgradients at every point in the original coordinate space to exactly 1. We achieve this by modifying the convex conjugate, $f_n^*$, of $f_n$. This is natural as the convex conjugate encodes gradient information.

Since $f_n$ is $1$-Lipschitz, we have that $\dom(f_n^*)\subseteq\B(0,1)$ and we can write
\begin{equation*}
    f_n(x) = f_n^{**}(x) = \sup_{\norm{p}\leq 1} \ip{p,x} - f_n^*(p).
\end{equation*}
We extend $f_n:\R^n\to\R$ to $g_n:\R^n\times\R\to\R$ by modifying this expression for $f_n$:
\begin{align*}
    g_n(x,y)&\coloneqq \sup_{\norm{p}\leq 1}\ip{\begin{pmatrix}
    p\\
    \sqrt{1-\norm{p}^2}
    \end{pmatrix}, \begin{pmatrix}
    x\\
    y
    \end{pmatrix}} - f_n^*(p).
\end{align*}
Finally, define $\phi_{n+1}(x) = g_n(\Pi_n x, x_{n+1})$.

In the case of $n=0$, the above construction degenerates (and is equivalent) to
\begin{align*}
    f_0:\R^0\to\R\quad&\text{maps}\quad f_0(0)= f(0)\\
    f_0^*:\R^0\to\R\quad&\text{maps}\quad f^*(0)= -f(0)\\
    g_0:\R^0\times\R\to\R\quad&\text{maps}\quad g_0(0,y) = y + f(0)\\
    \phi_1:\R^{N+1}\to\R\quad&\text{maps}\quad \phi_1(x) = x_1 + f(0)
\end{align*}

\begin{lemma}
\label{lem:lipschitz}
For all $n=0,\dots,N$,
$\phi_{n+1}:\R^{N+1}\to\R$ is a $1$-Lipschitz convex function.
\end{lemma}
\begin{proof}
By definition, $\phi_{n+1}$ is the pointwise supremum of $1$-Lipschitz affine functions and is thus $1$-Lipschitz and convex. It remains to rule out that $\phi_{n+1}(x)= +\infty$ for all $x$.
Note that
\begin{align*}
    \phi_{n+1}(0)&= \sup_{\norm{p}\leq 1}\inf_x f_n(x) - \ip{p,x} \leq \inf_x \sup_{\norm{p}\leq 1} f_n(x)- \ip{p,x}\\
    &\leq \sup_{\norm{p}\leq 1}f_n(0) - \ip{p,0} = f(0)  <\infty.\qedhere
\end{align*}
\end{proof}

The following expression for $f_n^*$ will be useful in bounding $\phi_{n+1}$ in the following lemma.
\begin{align*}
    f_n^*(p) &\coloneqq \sup_{x} \ip{p,x} - f_n(x)\\
    &= \sup_{x} \ip{p,x} - h\left(\norm{\Pi_n^* x - x_\star}\right)\\
    &= \sup_{x} \ip{p,x} - h\left(\sqrt{\norm{x - z_n}^2 + \Delta_n^2 }\right)\\
    &= \ip{p,z_n} + \sup_{\delta}\ip{p,\delta} - h\left(\sqrt{\norm{\delta}^2 + \Delta_n^2 }\right)\\
    &=\ip{p,z_n} + \sup_{\alpha\geq 0}\alpha\norm{p} - h\left(\sqrt{\alpha^2 + \Delta_n^2 }\right).
\end{align*}

\begin{lemma}
\label{lem:sharp}
For all $n=0,\dots,N$, it holds that $\phi_{n+1}(x_\star)\leq 0$.
\end{lemma}
\begin{proof}
If $n=0,\dots,N-1$, then
\begin{align}
    \phi_{n+1}(x_\star)&= g_n(z_n, -h(\Delta_n))\nonumber\\
    &= \sup_{\norm{p}\leq 1}\inf_{\alpha\geq 0} - h(\Delta_n)\sqrt{1-\norm{p}^2} - \alpha\norm{p} + h\left(\sqrt{\alpha^2 + \Delta_n^2}\right)\label{equation1}\\
    &= \sup_{0\leq \beta\leq 1}\inf_{\alpha\geq0}  - h(\Delta_n)\sqrt{1-\beta^2} - \alpha\beta + h\left(\sqrt{\alpha^2 + \Delta_n^2}\right)\nonumber\\
    &\leq \sup_{0\leq \beta\leq 1}\inf_{\alpha\geq0}  - h(\Delta_n)\sqrt{1-\beta^2} - \alpha\beta + 
    h(\Delta_n) + \sqrt{\alpha^2 + \Delta_n^2} - \Delta_n \nonumber\\
    &= \sup_{0\leq \beta\leq 1}\left(\Delta_n - h(\Delta_n)\right) \left(\sqrt{1-\beta^2} - 1\right) = 0.\nonumber
\end{align}
Here, the first inequality follows by upper bounding $h(\sqrt{\alpha^2+\Delta_n^2})$ by $1$-Lipschitzness of $h$.
The last line follows by minimizing in $\alpha$.

If $n = N$, then
\begin{align*}
    \phi_{n+1}(x_\star)&= g_n(z_n, -\Delta_n)\\
    &= \sup_{\norm{p}\leq 1}\inf_{\alpha\geq 0} - \Delta_n\sqrt{1-\norm{p}^2} - \alpha\norm{p} + h\left(\sqrt{\alpha^2 + \Delta_n^2}\right)\\
    &\leq \sup_{\norm{p}\leq 1}\inf_{\alpha\geq 0} - h(\Delta_n)\sqrt{1-\norm{p}^2} - \alpha\norm{p} + h\left(\sqrt{\alpha^2 + \Delta_n^2}\right).
\end{align*}
Here, the inequality follows as $h$ is $1$-Lipschitz. The remainder of the proof follows from the previous case upon recognizing \eqref{equation1}.
\end{proof}

The following lemma will be useful for bounding the performance of general subgradient-span methods.
\begin{lemma}
\label{lem:resisting}
The constructed function $f_\textup{hard}$ satisfies the ``zero-chain condition'': For any $x\in\R^{N+1}$ with $\supp(x)\subseteq [k]$, there exists $g\in \partial f_\textup{hard}(x)$ satisfying $\supp(g)\subseteq[k+1]$. In particular, there exists a resisting oracle for $x\mapsto g \in\partial f_\textup{hard}(x)$ that returns such a subgradient.
\end{lemma}
\begin{proof}
First, we claim that for all $n\geq k$, $\phi_{n+1}(x)= f(x)$.
\begin{align*}
    \phi_{n+1}(x) &= g_n(\Pi_n x, x_{n+1}) = g_n(\Pi_n x, 0)\\
    &= \sup_{\norm{p}\leq 1} \ip{p, \Pi_n x} - f_n^*(p)\\
    &= f_n(\Pi_n x) = f(\Pi_n^*\Pi_n x) = f(x).
\end{align*}
Here, the third line follows by Fenchel--Moreau.
We deduce that $\phi_{k+1}(x)=\phi_{k+2}(x) = \dots = \phi_{N+1}(x) = f(x)$.

Thus,
\begin{equation*}
    f_\textup{hard}(x)= \max\set{\phi_1(x),\phi_2(x),\dots,\phi_{N+1}(x), f(x)}
\end{equation*}
must have some active component $\phi_{n+1}(x)$ with $n\leq k$. Then, $\partial \phi_{n+1}(x)\subseteq \partial f_\textup{hard}(x)$. This completes the proof as $\phi_{n+1}$ depends only on the first $n+1\leq k+1$ coordinates of $x$.
\end{proof}

Collecting these lemmas together gives a proof of \cref{thm:lower}.
\begin{proof}[Proof of \cref{thm:lower}]
By \cref{lem:lipschitz,lem:sharp}, it holds that 
$f_\textup{hard}$ is $L$-Lipschitz, convex, with unique minimizer $x_\star$. Furthermore, $f_\textup{hard}(x)-f_\textup{hard}(x_\star) \geq f(x) = h(\norm{x-x_\star})$ holds on the neighborhood $\set{x:\, \norm{x - x_\star}\leq D}$. By \eqref{eq:distance}, it holds that $\norm{x_\star} \leq D$. Thus,
$(f_\textup{hard},x_0=0)\in\cF_{h,L,D}$.

Let $x_0,x_1,\dots$ denote a sequence of iterates produced by some subgradient-span method on this instance against the resisting oracle from \cref{lem:resisting}.
We claim by induction that $\supp(x_n)\subseteq [n]$ for all $n$. The base case $\supp(x_0) = \emptyset$ holds.
Now, suppose $\supp(x_n)\subseteq [n]$ for all $n \leq n_0$. By \cref{lem:resisting}, $\supp(g_n)\subseteq[n+1]$ for all $n\leq n_0$. Thus, by the subgradient-span condition, $\supp(x_{n_0+1}) \subseteq [n_0+1]$.

Finally, since $\supp(x_n)\subseteq[n]$, we deduce that
\begin{equation*}
    \norm{x_n - x_\star} \geq \inf_{x':\, \supp(x')\subseteq [n]} \norm{x' - x_\star} = \norm{z_n - x_\star} = \Delta_n.\qedhere
\end{equation*}
\end{proof}

Since the minimum value of $f_\textup{hard}$ on the coordinate space $\set{x:\, \supp(x)\subseteq[n]}$ is $f_\textup{hard}(z_n)=h(\Delta_n)$, we deduce the following corollary to \cref{thm:lower}.
\begin{corollary}
Suppose $d\geq N+1$.
Suppose $L,D>0$ and $h:[0,D]\to\R_+$ is $L$-Lipschitz, convex, increasing, and vanishes at 0. There exists an instance $(f_\textup{hard},x_0=0)\in \cF_{h,L,D}$ with $f_\textup{hard}(x_\star) = 0$ and a resisting oracle for $x\mapsto g\in\partial f_\textup{hard}(x)$ so that for any choice of $x_1,\dots, x_N$ satisfying the subgradient-span condition, $x_n\in x_0 + \spann\set{g_0,\dots,g_{n-1}}$, it holds that
\begin{equation*}
    f_\textup{hard}(x_n) - f_\textup{hard}(x_\star) \geq h(\Delta_n) \qquad\forall n=0,\dots,N.
\end{equation*}
\end{corollary}

In the H\"olderian error bound setting, i.e., $h(t) = ct^{1/\theta}$, we may replace $\Delta_n$ with its asymptotic expression (see \cref{lem:heb_asymptotics}) to deduce that at least
\begin{equation*}
    n\geq N_\epsilon \equiv \begin{cases}
    \frac{\theta}{1-\theta}\frac{L^2}{c^{2\theta}}\epsilon^{-2(1-\theta)} & \text{if }\theta<1\\
    2 \frac{\log(cD/\epsilon)}{\log(1/(1- c^2/L^2))} &\text{if } \theta = 1
    \end{cases}
\end{equation*}
iterations is required to guarantee $f_n-f_\star \leq \epsilon$ on the class $\cF_{h,L,D}$.
Here,
$N_\epsilon = \min\set{n:\, h(\Delta_n)\leq \epsilon}$ and $\equiv$ denotes equivalence up to lower order terms as $\epsilon\to0$.
This matches the guarantees provided by the restarting strategy for $\cF_{h,L,D}$ up to constants (see \cite{roulet2017sharpness,yang2018rsg}) and essentially recovers the bounds stated by Nemirovski and Nesterov~\cite[Equation 1.21]{nemirovski1985optimal} for H\"older-smooth convex optimization.

\section{Open questions}
We leave open the following interesting extensions:
\begin{itemize}
    \item \textbf{Minimax optimal methods in the non-Euclidean setting}: Restarted mirror descent~\cite{ding2023sharpness} provides accelerated convergence guarantees where growth is measured in an $\ell_p$ norm. 
The lower bounds presented in \cref{sec:lower} are fundamentally Euclidean and are unlikely to extend to the non-Euclidean setting.
    \item \textbf{Subgame perfect methods}: Recent work~\cite{grimmer2025subgame,drori2016optimal,grimmer2024beyond} has formalized the ``subgame perfect'' notion of dynamic optimality (a stronger notion than minimax optimality) and exhibited such methods and matching lower bounds for smooth and nonsmooth convex minimization.
    In both cases, the subgame perfect methods are dynamic reoptimizations of known minimax optimal methods.
    We leave the extension of the minimax optimal \Polyak{} and \Decay{} and their matching lower bounds in the static setting to the dynamic setting for future work.
\end{itemize}

\section*{Acknowledgments}
We would like to thank Lijun Ding for helpful discussions.

{
    \small
    \bibliographystyle{plainnat}

}

\appendix

\section{Proof of \cref{lem:heb_asymptotics}}

Throughout this appendix, $\equiv$ denotes equality up to lower order terms as $n\to\infty$.

The case $\theta = 1$ is straightforward. Thus, suppose $L,D>0$ and that $h(t) = c t^{1/\theta}$ for some $c>0$ and $\theta\in(0,1)$. Assume that $h$ is $L$-Lipschitz on $[0,D]$.
As $h$ is $L$-Lipschitz with $\theta < 1$, we deduce that $h(D)<LD$ so that $D^2 - \frac{h(D)^2}{L^2}> 0$. 

Our first step is to reparameterize the inductive update by defining $x_n = \Delta_n^2$:
\begin{equation*}
    x_0 = D^2\qquad\text{and}\qquad
    x_{n+1} = \Delta_{n+1}^2 = \Delta_n^2 - \frac{h(\Delta_n)^2}{L^2} = x_n - \frac{c^2}{L^2}x_n^{1/\theta}.
\end{equation*}
Let $a = 1/\theta>1$ and $\alpha = c^2/L^2 >0$ so that $x_{n+1} = x_n - \alpha x_n^a$. From above, it holds that $x_0 - \alpha x_0^a>0$.

\begin{lemma}
Suppose a sequence $x_0,x_1,\dots,$ is defined recursively by $x_{n+1} = x_n - \alpha x_n^a$ where $a>1$, $\alpha >0$ and 
$x_0>0$ satisfies $x_0 - \alpha x_0^a > 0$. Then, the sequence grows asymptotically as:
\begin{equation*}
    x_n \equiv \left[\alpha(a-1)n\right]^{-1/(a-1)}.
\end{equation*}
\end{lemma}
\begin{proof}
First, we check that $x_n$ is well-defined, decreasing, and tending to zero.
By assumption, $x_0 \in (0, \alpha^{-1/(a-1)})$.
Then, by the formula $x_{n+1} = x_n(1 - \alpha x_n^{a-1})$, we deduce that $x_{n}$ is a decreasing positive sequence.
Let $x_\infty = \lim_{n\to\infty}x_n \in[0,\alpha^{-1/(a-1)})$. Passing to limits in the formula $x_{n+1} = x_n(1 - \alpha x_n^{a-1})$ shows $x_\infty = x_\infty(1 - \alpha x_\infty^{a-1})$ and that $x_\infty = 0$.

Let $\phi(x) = x^{-(a-1)}$.
We will prove that the sequence $\phi(x_0),\phi(x_1),\dots$ is asymptotically linear.
The function $\phi$ is $C^\infty$ on the positive line. 
Let $\delta = \alpha x_{n-1}^{a}$.
By Taylor's Theorem,
\begin{equation*}
    \phi(x_n) - \phi(x_{n-1}) = \phi(x_{n-1} - \delta) - \phi(x_{n-1})= \left[-\delta \phi'(x_{n-1})\right] + \left[\phi''(x)\frac{\delta^2}{2}\right]
\end{equation*}
for some $x\in[x_n, x_{n-1}]$. 
The first square-bracketed term is exactly $(a-1)\alpha$.
We now upper bound the second square-bracketed term. Note that
\begin{equation*}
    \frac{x_n}{x_{n-1}} = 1 - \alpha x_{n-1}^{a-1} \geq 1 - \alpha x_{0}^{a-1} >0.
\end{equation*}
Thus,
\begin{align*}
    \phi''(x)\frac{\delta^2}{2} &= a(a-1)x^{-(a+1)}\frac{\left(\alpha x_{n-1}^a\right)^2}{2}\leq \frac{\alpha^2a(a-1)}{2}\left(\frac{x_{n-1}}{x_n}\right)^{a+1}x_{n-1}^{a-1}\\
    &\leq \frac{\alpha^2a(a-1)}{2}\left(\frac{1}{1 - \alpha x_{0}^{a-1}}\right)^{a+1}x_{n-1}^{a-1} = o(1).
\end{align*}
Plugging this back in gives
\begin{align*}
    \phi(x_N) - \phi(x_0) &= \sum_{n=1}^N\phi(x_n) - \phi(x_{n-1})\\
    &\leq N(a-1)\alpha + \sum_{n=1}^N o(1)\\
    &= N(a-1)\alpha + o(N).
\end{align*}
We conclude that $\phi(x_n) \equiv \alpha(a-1)n$. Translating back to the $x_n$ sequence completes the proof.
\end{proof} 
\end{document}